\documentclass{amsart}
\usepackage[english]{babel}
\usepackage{amsthm}
\usepackage{inputenc}
\usepackage{amssymb}
\usepackage{graphicx}
\usepackage{color}
\def\bc{\begin{center}}
\def\ec{\end{center}}
\def\d{\displaystyle}

\newtheorem{thm}{Theorem}[section]

\newtheorem{lem}[thm]{Lemma}

\newtheorem{defn}[thm]{Definition}

\numberwithin{equation}{section}

\def\ll{\mathcal{Z}}
\begin{document}

\title[Naturally graded Zinbiel algebras with nilindex $n-3$]{Naturally graded Zinbiel algebras with nilindex $n-3$}%

\author{J.Q. Adashev, L.M. Camacho, S. G\'{o}mez-Vidal,
 I.A. Karimjanov}
\address{[J.Q. Adashev -- I.A. Karimjanov] Institute of Mathematics and Information Technologies
 of Academy of Uzbekistan, 29, F.Hodjaev srt., 100125, Tashkent (Uzbekistan)}
\email{adashevjq@mail.ru --- iqboli@gmail.com}
\address{[L.M. Camacho] Dpto. Matem\'{a}tica Aplicada I.
Universidad de Sevilla. Avda. Reina Mercedes, s/n. 41012 Sevilla.
(Spain)} \email{lcamacho@us.es --- samuel.gomezvidal@gmail.com}

%

\begin{abstract}
We present the classification of a subclass of $n$-dimensional naturally graded Zinbiel algebras. This subclass has the nilindex $n-3$ and the characteristic sequence $(n-3,2,1).$
In fact, this result completes the classification of naturally graded Zinbiel algebras of nilindex $n-3.$
\end{abstract}
\maketitle


\section{Introduction.}

Intensive investigation on Lie algebras leads to the appearance of a new algebraic object -- Leibniz algebras. The Leibniz algebras introduced by Loday in \cite{loday1} are a ''non commutative" algebras analogue to Lie algebras. It should be mentioned that Leibniz algebras inherit an important Lie algebra property: the operator of right multiplication on an element of an algebra is a derivation.

Leibniz algebras form a Koszul operad in the sense of V. Ginzburg and M. Kapranov \cite{Gin}. Under the Koszul duality the operad of Lie algebras is dual the operad of associative and commutative algebras. The notion of dual Leibniz algebra defined by J.-L. Loday \cite{Loday} is precisely the dual operad of Leibniz algebras in this sense.

In this paper, we study algebras which are the dual to Leibniz algebras in Koszul sense. J.-L. Loday studied in \cite{Loday} categorical properties of Leibniz algebras and considered in this connection a new object -- Zinbiel algebras (Leibniz is written in reverse order). Since the category of Zinbiel algebras is Koszul dual to the category of Leibniz algebras, sometimes they are also called dual Leibniz algebras.

In \cite{Jobir3, Dz, Umir} some crucial properties of Zinbiel algebras were obtained. Particularly, in \cite{Dz}, the authors prove that every finite-dimensional Zinbiel algebra over complex numbers is nilpotent. However, the study of nilpotent algebras is too complex and should be carried out with additional conditions, such as conditions on nilindex, various types of gradations, characteristic sequence and others.

The aim of this work is to continue the study of complex
finite-dimensional naturally graded Zinbiel algebras. The n-dimensional Zinbiel algebras of nilindex $k$ with $n-2\leq k\leq n$ are classified in
\cite{TesisJobir, Jobir3}. The
classification of complex $n$-dimensional naturally graded Zinbiel algebras of nilindex $n-3$ is a difficult problem and it should be divided into three cases. Namely, it is necessary to consider the possibilities of the characteristic
sequence of such algebras: $(n-3,3)$, $(n-3, 1, 1, 1)$ and $(n-3,
2, 1).$ The classification of complex naturally graded Zinbiel
algebras of nilindex $n-3$ with characteristic sequence equal to
$(n-3,3)$ and $(n-3, 1, 1, 1)$ has been done in \cite{TesisJobir}.

The knowledge of naturally graded algebras of a certain family offers significant
information about their structural properties.

In this paper we obtain the classification of naturally graded
Zinbiel algebras of nilindex $n-3$ with characteristic sequence
$(n-3,2,1).$ Thus, we complete the study for the $n-3$ case.
All the spaces and the algebras are considered over the field of complex numbers.
We omit the products which are equal to zero for convenience.

Throughout all the work we use the software $Mathematica$ (see \cite{JSC})
to compute the Zinbiel identity in low dimensions and to formulate the generalizations of the calculations,
which are proved for arbitrary dimension.
Moreover, the program allows us to construct new bases using some general transformation of the generators of the algebra.

Since the direct sum of nilpotent Zinbiel algebras is nilpotent, we shall consider only non split algebras.

\section{Preliminaries}
In this section we introduce some definitions, notations and
results, which are necessary for the understanding of graded Zinbiel algebras.

\begin{defn} A vector space $\ll$ over a field $K$ with a bilinear operation ``$\circ$'' is called Zinbiel
algebra if for any $x,y,z \in \ll$ the following identity
\begin{equation}\label{21}
(x\circ y)\circ z=x\circ (y\circ z)+x\circ (z \circ y)
\end{equation}
holds.
\end{defn}

Examples of Zinbiel algebras can be found in \cite{Jobir3,Dz,Loday}.

$Z(a,b,c)$ denotes the following polynomial:
$$Z(a,b,c)=(a\circ b)\circ c-a\circ(b\circ c)-a\circ(c\circ b).$$

Zinbiel algebras are defined by the identity
$Z(a,b,c)=0.$

For a given Zinbiel algebra $\ll$ the sequence of two-sided ideals defined recursively as follow:
$$\ll^1=\ll, \ \ll^{k+1}=\ll \circ \ll^k, \ k\geq 1. $$
is said to be the lower central series.

\begin{defn} A Zinbiel algebra $\ll$ is called nilpotent if there exists $s \in \mathbb{N}$ such that $\ll^s \neq 0$ and $\ll^{s+1}=0.$ The minimal number $s$ satisfying this property is called the index of nilpotency or nilindex of the algebra $\ll.$
\end{defn}

For a given Zinbiel algebra $\ll$ we introduce denotations:\\
$$R(\ll) = \{x \in \ll \ | \  y \circ x = 0 \hbox{ for any } y \in
\ll\}\ -- \ \mbox{ \emph{the right annihilator} of }\ll,$$
$$L(\ll) = \{x \in \ll \ | \ x \circ y = 0 \hbox{ for any } y \in
\ll\}\ --\ \mbox{ \emph{the left annihilator} of }\ll,$$
$$Cent(\ll) = \{x,y \in \ll \ | \ x \circ y= y \circ x = 0 \hbox{
for any } y \in \ll\}\ --\ \mbox{ \emph{the center} of }\ll.$$

It is easy to see that the center and the right annihilator of $\ll$ are two-sided ideals.

Let us denote by $L_{x}$ the operator of left multiplication on
element $x$, i.e. $L_x: \ll \longrightarrow \ll$ such that
$L_x(y)=x\circ y$ for any $y \in \ll.$

Let $\ll$ be a complex $n$-dimensional Zinbiel algebra and $x$ be
an element of the set $\ll \setminus \ll^2$. For the operator
$L_x$ we define a descending sequence $C(x)=(n_1, n_2, \dots,
n_k)$ with $n_1+\cdots+n_k=n$, which consists of the dimensions of the Jordan blocks of
the operator $L_x$. In the set of such sequences we consider the
lexicographic order, that is, $C(x)=(n_1,n_2, \dots, n_k) < C(y)=(m_1, m_2, \dots, m_s)$ if there exists $i$
such that $n_i<m_i$ and $n_j=m_j$ for $j<i$. Taking into account the equality $n_1+\dots+n_k=m_1+\dots+m_s$ such comparison is always applicable.

  \begin{defn}\label{def:char.seq}
  The sequence $C(\ll)=\max \{C(x)\  : \ x \in \ll \setminus \ll^2 \}$ is called the characteristic sequence
  of the algebra $\ll$.
  \end{defn}

In \cite{Dz}, the authors prove that Zinbiel algebras of finite dimension are nilpotent.
Since we focused our attention on finite dimension complex nilpotent Zinbiel
algebras.

Let $\ll$ be a finite-dimensional nilpotent Zinbiel algebra with
nilindex equal to $s$. For $i$ ($1\leq i\leq s)$ we put
$\ll_i=\ll^i/\ll^{i+1}$ and we obtain the graded Zinbiel algebra
$$gr(\ll) = \ll_1 \oplus
\ll_2\oplus \ldots\oplus \ll_{s},\ \mbox{ where } \ll_i\circ \ll_j \subseteq \ll_{i+j}.$$

An algebra $\ll$ if called naturally graded if $\ll\cong gr(\ll).$ It is not difficult to see that
$\ll_{i+1}=\ll_1 \circ \ll_i$ in the naturally graded algebra $\ll.$

Let $\ll$ be a naturally graded Zinbiel algebra with
characteristic sequence $(n-3,2,1).$ By de\-finition of
characteristic sequence there exists a basis $\{e_1, e_2, \dots,
e_n\}$ in the algebra $\ll$ such that the operator $L_{e_1}$ has
one block $J_{n-3}$ of size $(n-3),$ one block $J_2$ of size $2$ and one block $J_1$ of size one.

Note that there will be six possibilities for the operators $L_{e_1}.$
By a change of basis it is easy to prove that the six cases can be reduced to the following three cases:
$$
I.\left(\begin{array}{ccc}
J_{n-3}&0 &0 \\
0&J_2 &0\\
0& 0& J_1
\end{array} \right),
\ \ \ II.\left(\begin{array}{ccc}
J_2&0 &0 \\
0&J_{n-3} &0\\
0& 0& J_1
\end{array} \right),
\ \ \ III.\left(\begin{array}{ccc}
J_1&0 &0 \\
0&J_{n-3} &0\\
0& 0& J_2
\end{array} \right).$$

\begin{defn} A Zinbiel algebra $\ll$ is called either of first type (type I), second type (type II) or third type (type III) if the operator $L_{e_1}$ has the form:
$$
I.\left(\begin{array}{ccc}
J_{n-3}&0 &0 \\
0&J_2 &0\\
0& 0& J_1
\end{array} \right),
\ \ \ II.\left(\begin{array}{ccc}
J_2&0 &0 \\
0&J_{n-3} &0\\
0& 0& J_1
\end{array} \right),
\ \ \ III.\left(\begin{array}{ccc}
J_1&0 &0 \\
0&J_{n-3} &0\\
0& 0& J_2
\end{array} \right)$$
respectively.

\end{defn}

From now on we denote by $C_i^j$ the combinatorial numbers $C_i^j=\left (\begin{array}{l}
i\\
j
\end{array} \right ).$

The following result holds:
\begin{lem}\cite{p-FNGZA}
Let $\ll$ be a Zinbiel algebra such that $e_1\circ e_i=e_{i+1}$ for $1\leq i\leq k-1,$ with respect to the adapted basis $\{e_1,\dots,e_{k},e_{k+1},\dots,e_n\}.$ Then
$$e_i\circ e_j=C_{i+j-1}^j e_{i+j},\quad \mbox{ for }\quad 2\leq i+j\leq k$$
\end{lem}

\section{Main Result}

\subsection{Type I} Algebras of type I with $n\geq 8.$ So, we have the following brackets:
$$\left\{\begin{array}{ll}
\,e_1\circ e_i = e_{i+1},& 1\leq i\leq n-4,\\
\,e_1\circ e_{n-3} =0,&\\
\,e_1\circ e_{n-2} =e_{n-1},&\\
\,e_1\circ e_{n-1} =0,&\\
\,e_1\circ e_n =0.&\\
 \end{array}\right.$$

It is easy to see that $ \ll_i\supseteq \langle e_i\rangle $ where  $1 \leq i \leq
n-3$. It is evident that $dim(\ll_1)>1.$ In fact, if $dim(\ll_1)=1,$ then the algebra $\ll$ is one-degenerated and therefore it is a zero-filiform algebra, but it is not an algebra of nilindex $n-3.$ Let us assume that $e_{n-2}\in \ll_{r_1}$ and $e_n\in \ll_{r_2}$, then $e_{n-1}\in \ll_{r_1+1}$.

We can distinguish the following cases:

%
%

\noindent \textbf{Case I.} If $r_1=r_2=1$.

Then we have that
$$\ll_1=<e_1,e_{n-2},e_n>,\ \ll_2=<e_2,e_{n-1}>,\ \ll_3=<e_3>,\dots, \ll_{n-3}=<e_{n-3}>$$
and the following products:
$$\begin{array}{lll}
\, e_1\circ e_1 = e_2,& e_1\circ e_{n-2} = e_{n-1}, & e_{n-2}\circ e_1 =\alpha_1 e_2 + \alpha_2 e_{n-1},\\
\, e_{n-2}\circ e_{n-2} =\alpha_3 e_2 + \alpha_4 e_{n-1}, & e_{n-2}\circ e_n =\alpha_5 e_2 + \alpha_6 e_{n-1}, & e_n\circ e_1= \beta_1e_2+ \beta_2 e_{n-1},\\
\, e_n\circ e_{n-2} = \beta_3e_2+ \beta_4 e_{n-1}, & e_n\circ e_n =\beta_5 e_2 + \beta_6 e_{n-1}, & e_1\circ e_2 = e_3,\\
\, e_{n-2}\circ e_2 =\gamma_1 e_3,& e_{n-2}\circ e_{n-1} = \gamma_2e_3, & e_n\circ e_2 =\gamma_3 e_3,\\
\, e_n\circ e_{n-1} =\gamma_4 e_3.\\
\end{array}$$

From the equality
$Z(e_1,e_n,e_1)=Z(e_1,e_n,e_n)=0$
we have $\beta_1=\beta_5=0$.

Let us consider the equalities $Z(e_1,e_{n-2},e_1)=Z(e_1,e_{n-1},e_1)=0$
then it follows $\alpha_1=0$.

From the equalities
$$\begin{array}{l}
Z(e_1,e_1,e_{n-2})=Z(e_{n-2},e_1,e_1)=Z(e_{n-2},e_{n-1},e_1) =Z(e_1,e_1,e_n)=0\\
Z(e_1,e_n,e_2)=Z(e_n,e_{n-1},e_1)=Z(e_1,e_{n-2},e_{n-2})=Z(e_1,e_{n-1},e_{n-2})=0\\
Z(e_1,e_{n-2},e_n)=Z(e_1,e_n,e_{n-1})=Z(e_1,e_1,e_{n-1})=Z(e_1,e_{n-2},e_2)=0\\
Z(e_{n-2},e_n,e_1)= Z(e_{n-2},e_{n-2},e_1)=0\\
\end{array}$$
we obtain
$$\gamma_1=\gamma_2=\gamma_3=\gamma_4=\alpha_3=\alpha_5=\beta_3=0,$$
and
$$e_2\circ e_{n-2}=e_{n-2}\circ e_2=e_2\circ e_{n-1}=e_{n-1}\circ e_2=e_2\circ e_n=e_n\circ e_2=0.$$

Now, by mathematical induction method, we prove that $e_{n-1}\circ e_k=0$ and $e_k\circ e_{n-1}=0$ with $2\leq k\leq n-3.$
\begin{itemize}
\item If $k=2,$ then we have $e_{n-1}\circ e_2 = e_2\circ e_{n-1}=0$.

\item Let us suppose that for some $k$ the equalities $e_{n-1} \circ e_k = 0$ and $e_k\circ e_{n-1}=0$ are true.
We prove it for $k+1$.
$$\begin{array}{ll}
e_{n-1}\circ e_{k+1}&=e_{n-1}\circ (e_1\circ e_k)=(e_{n-1}\circ
e_1)\circ e_k - e_{n-1}\circ (e_k\circ e_1)=\\
&=-C_k^1 e_{n-1}\circ e_{k+1}= -ke_{n-1}\circ e_{k+1},\quad e_{n-1}\circ e_{k+1}=0.\\[2mm]
e_{k+1}\circ e_{n-1}&=(e_1\circ e_k)\circ e_{n-1}=e_1\circ (e_{k}\circ
e_{n-1})+e_1\circ (e_{n-1}\circ e_k)=\\
&=0
\end{array}$$

\end{itemize}

\

As in previous cases, it easy to see that
$e_{k}\circ e_{n-2}=e_{n-2}\circ e_k=0$ and $e_{k}\circ e_{n}=e_{n}\circ e_k=0$ for $2\leq k\leq n-3.$

Thus, we have obtained the following family of algebras:
$$Z(a_1,a_2,a_3,a_4,a_5,a_6):\left\{\begin{array}{ll}
\,e_i\circ e_j= C_{i+j-1}^j e_{i+j},& 2\leq i+j\leq n-3, \\
e_1\circ e_{n-2}=e_{n-1},& \\
e_{n-2}\circ e_1=a_1e_{n-1},&\\
\,e_{n-2}\circ e_{n-2}=a_2e_{n-1},&\\
 e_{n-2}\circ e_n=a_3e_{n-1}, &\\
  e_n\circ e_1=a_4e_{n-1},&\\
\, e_n\circ e_{n-2}=a_5e_{n-1},& \\
e_n\circ e_n=a_6e_{n-1},&
\end{array}\right.$$
where we omit the products that are equal to zero.

\begin{thm}\label{teorema1}
An arbitrary Zinbiel algebra of the family $Z(a_1,a_2,a_3,a_4,a_5,a_6)$ is isomorphic to one of the following pairwise non-isomorphic algebras:
$$\begin{array}{lll}
Z_1(1,0,0,0,1,0), &Z_2(0,0,0,0,1,0),&Z_3(0,1,0,1,0,0),\\
Z_4(0,0,0,1,0,0),&Z_5(0,1,0,0,0,0),&Z_6(1,1,0,0,0,0),\\
 Z_7(\lambda,0,0,0,0,0),\ \lambda\in \mathbb{C},&Z_8(0,\lambda,1,0,0,1),\ \lambda \in \mathbb{C}\setminus\{0\},&Z_9( \alpha,-\frac{\alpha}{(\alpha-1)^2},1,0,0,1),\  \alpha \in \mathbb{C}\setminus\{0,1\},\\
 Z_{10}(0,0,1,0,1,1),&  Z_{11}(1,0,1,0,1,1),&Z_{12}(0,0,1,1,0,0),\\
 Z_{13}(0,0,1,0,0,0),&Z_{14}(\lambda,1,1,0,1,1),\ \lambda\in \mathbb{C},& Z_{15}(0,1,1,-1,1,1),\\
Z_{16}(1,1,1,0,1,1).& & \\
\end{array}$$
\end{thm}

\begin{proof} Let $\ll$ be satisfying to the hypothesis of the theorem. Due to the property of natural gradation of the algebra it is enough to consider the following change of generators:
$$\begin{array}{rl}
\,e_1'&=P_1e_1+P_{n-2}e_{n-2}+P_ne_n,\\
\,e_{n-2}'&=Q_1e_1+Q_{n-2}e_{n-2}+Q_ne_n,\\
\,e_n'&=R_1e_1+R_{n-2}e_{n-2}+R_ne_n.\\
\end{array}$$

Making the general change of basis in the family $Z(a_1,a_2,a_3,a_4,a_5,a_6),$ we derive the expressions of the new parameters in the new basis $(1)$:
$$\begin{array}{l}
a_1'=\displaystyle\frac{a_1P_1Q_{n-2}+a_2P_{n-2}Q_{n-2}+a_3P_nQ_{n-2}+a_4P_1Q_n+a_5P_{n-2}Q_n+a_6P_nQ_n}{P_1Q_{n-2}+a_2P_{n-2}Q_{n-2}+a_3P_{n-2}Q_n+a_5P_nQ_{n-2}+a_6P_nQ_n},\\[4mm]
a_2'=\displaystyle\frac{a_2Q_{n-2}^2+a_3Q_{n-2}Q_n+a_5Q_{n-2}Q_n+a_6Q_n^2}{P_1Q_{n-2}+a_2P_{n-2}Q_{n-2}+a_3P_{n-2}Q_n+a_5P_nQ_{n-2}+a_6P_nQ_n},\\[4mm]
a_3'=\displaystyle\frac{a_2Q_{n-2}R_{n-2}+a_3Q_{n-2}R_n+a_5Q_nR_{n-2}+a_6Q_nR_n}{P_1Q_{n-2}+a_2P_{n-2}Q_{n-2}+a_3P_{n-2}Q_n+a_5P_nQ_{n-2}+a_6P_nQ_n},\\[4mm]
a_4'=\displaystyle\frac{a_1P_1R_{n-2}+a_2P_{n-2}R_{n-2}+a_3P_nR_{n-2}+a_4P_1R_n+a_5P_{n-2}R_n+a_6P_nR_n}{P_1Q_{n-2}+a_2P_{n-2}Q_{n-2}+a_3P_{n-2}Q_n+a_5P_nQ_{n-2}+a_6P_nQ_n},\\[4mm]
a_5'=\displaystyle\frac{a_2Q_{n-2}R_{n-2}+a_3Q_nR_{n-2}+a_5Q_{n-2}R_n+a_6Q_nR_n}{P_1Q_{n-2}+a_2P_{n-2}Q_{n-2}+a_3P_{n-2}Q_n+a_5P_nQ_{n-2}+a_6P_nQ_n},\\[4mm]
a_6'=\displaystyle\frac{a_2R_{n-2}^2+a_3R_{n-2}R_n+a_5R_{n-2}R_n+a_6R_n^2}{P_1Q_{n-2}+a_2P_{n-2}Q_{n-2}+a_3P_{n-2}Q_n+a_5P_nQ_{n-2}+a_6P_nQ_n},
\end{array}$$
and the following restrictions:
$$(2)\left\{\begin{array}{l}
Q_1=R_1=0,\\[1mm]
P_1R_{n-2}+a_2P_{n-2}R_{n-2}+a_3P_{n-2}R_n+a_5P_nR_{n-2}+a_6P_nR_n=0, \\[1mm]
P_1Q_{n-2}+a_2P_{n-2}Q_{n-2}+a_3P_{n-2}Q_n+a_5P_nQ_{n-2}+a_6P_nQ_n\neq 0,\\[1mm]
P_1(Q_{n-2}R_n-Q_nR_{n-2})\neq 0.
\end{array}\right.$$

We can distinguish two cases:

\

\noindent \textbf{Case 1.} Let $e_n\in R(Z)$ be, then $a_3=a_6=0.$

From the restrictions,
$$\left.\begin{array}{l}
(P_1+a_2P_{n-2}+a_5P_n)R_{n-2}=0,\\
(P_1+a_2P_{n-2}+a_5P_n)Q_{n-2}\neq0,\\
P_1(Q_{n-2}R_n-Q_nR_{n-2})\neq0.
\end{array}\right\} \Rightarrow R_{n-2}=0.$$
it follows that $P_1Q_{n-2}R_n\neq0.$
Thus, the new parameters are:
$$\begin{array}{l}
a_1'=\displaystyle\frac{a_1P_1Q_{n-2}+a_2P_{n-2}Q_{n-2}+a_4P_1Q_n+a_5P_{n-2}Q_n}{Q_{n-2}(P_1+a_2P_{n-2}+a_5P_n)},\\[4mm]
a_2'=\displaystyle\frac{a_2Q_{n-2}+a_5Q_n}{P_1+a_2P_{n-2}+a_5P_n},\\[4mm]
a_4'=\displaystyle\frac{R_n(a_4P_1+a_5P_{n-2})}{Q_{n-2}(P_1+a_2P_{n-2}+a_5P_n)},\\[4mm]
a_5'=\displaystyle\frac{a_5R_n}{P_1+a_2P_{n-2}+a_5P_n},
\end{array}$$

We observe that the nullity of $a_5$ is invariant. Moreover,
it is easy to check that the nullity of the following expression
$$a_2'a_4'-a_1'a_5'=\frac{(a_2a_4-a_1a_5)P_1R_n}{(P_1+a_2P_{n-2}+a_5P_n)^2}$$
is invariant. Thus, we can distinguish the following non-isomorphic cases:

\

\noindent \fbox{\textbf{Case 1.1.}} Let $a_5\neq0$ be. Then choosing
$$R_n=\frac{P_1+a_2P_{n-2}+a_5P_n}{a_5}, \quad P_{n-2}=-\frac{a_4P_1}{a_5}, \quad Q_n=-\frac{a_2Q_{n-2}}{a_5}$$
we have
$$a_5'=1, \quad a_4'=0, \quad a_2'=0, \quad a_1'=\frac{(a_2a_4-a_1a_5)P_1}{(a_2a_4-a_5)P_1-a_5^2P_n}$$
and the determinant is formed by the potencies of the following non-zero factors: $P_1Q_{n-2}a_5((a_2a_4-a_5)P_1-a_5^2P_n)$.

\begin{itemize}
\item[$a)$] If $a_2a_4-a_1a_5\neq0$, choosing $P_n=\frac{P_1(a_1-1)}{a_5}$ we receive $a_1'=1.$ It follows the algebra $Z_1(1,0,0,0,1,0).$\\

\item[$b)$] If $a_2a_4-a_1a_5=0$, then we obtain $a_1'=0$ and we have the algebra $Z_2(0,0,0,0,1,0).$
\end{itemize}

\

\noindent \fbox{\textbf{Case 1.2.}} Let $a_5=0$ be. Then, $a_5'=0$ and we have
$$\begin{array}{l}
a_1'=\displaystyle\frac{a_1P_1Q_{n-2}+a_2P_{n-2}Q_{n-2}+a_4P_1Q_n}{Q_{n-2}(P_1+a_2P_{n-2})},\\[4mm]
a_2'=\displaystyle\frac{a_2Q_{n-2}}{P_1+a_2P_{n-2}},\\[4mm]
a_4'=\displaystyle\frac{a_4P_1R_n}{Q_{n-2}(P_1+a_2P_{n-2})}.
\end{array}$$
with $P_1Q_{n-2}R_n(P_1+a_2P_{n-2})\neq0.$

We observe that the nullities of $a_2$ and $a_4$ are invariant, so we can distinguish the following cases:

\

\noindent $a)$ Let $a_4\neq0$ be. Then, choosing
$$R_n=\frac{Q_{n-2}(P_1+a_2P_{n-2})}{a_4P_1}, \quad Q_n=-\frac{Q_{n-2}(a_1P_1+a_2P_{n-2})}{a_4P_1},$$
we get $a_4'=1 $ and $ a_1'=0.$

\begin{itemize}
\item[$a.1)$] If $a_2\neq0$, then choosing $Q_{n-2}=\displaystyle\frac{P_1+a_2P_{n-2}}{a_2},$ we obtain $a'_2=1$ and the algebra $Z_3(0,1,0,1,0,0).$\\

\item[$a.2)$] If $a_2=0$, then we have $a_2'=0$ and the algebra $Z_4(0,0,0,1,0,0).$
\end{itemize}

\

\noindent $b)$ Let $a_4=0$ be. Then $a_4'=0$ and we have
$$a_1'=\frac{a_1P_1+a_2P_{n-2}}{P_1+a_2P_{n-2}}, \quad a_2'=\frac{a_2Q_{n-2}}{P_1+a_2P_{n-2}}.$$

We have that the nullity of the following expression:
$$a_1'-1=\frac{P_1(a_1-1)}{P_1+a_2P_{n-2}}.$$
is invariant.

\

\begin{itemize}
\item[$b.1)$] Let $a_2\neq0$ be. Then, choosing $Q_{n-2}=\displaystyle\frac{P_1+a_2P_{n-2}}{a_2},$ we obtain $a_2'=1.$

\noindent $\bullet$ If $a_1-1\neq0,$ then putting $P_{n-2}=-\frac{a_1P_1}{a_2},$ we have $a_1'=0$ and the algebra $Z_5(0,1,0,0,0,0)$. The determinant of change of basis consists of the potencies of the following non-zero factors $a_2(a_1-1)P_1R_n.$

\

\noindent $\bullet$ If $a_1-1=0,$ then $a_1'=1$ and we obtain $Z_6(1,1,0,0,0,0).$

\

\item[$b.2)$] Let $a_2=0$ be. Then, we have $a_2'=0, \ a_1'=a_1=\lambda\in \mathbb{C}$ and the family $Z_7(\lambda,0,0,0,0,0),$ with $ \lambda\in \mathbb{C}.$
\end{itemize}

\


\noindent\textbf{Case 2.} Let $e_n\notin R(Z)$ be, then $(a_3,a_6)\neq(0,0).$ We can suppose that $a_3\neq0,$ in another case, $a_3=0$ and $a_6\neq 0$ we make the following change of basis $f_1'=f_1+f_3.$ Thus, $a_3\neq 0.$ Taking into account the expressions given in $(1)$, the restrictions $(2)$ and the following expression:
$$\begin{array}{ll}
\Delta&=a_3^3a_4+a_3^2a_4a_5-a_1a_3^2a_4a_5+a_2a_3a_4^2a_5-a_1a_3a_4a_5^2-\\[1mm]
&-a_1a_3^2a_6-3a_2a_3a_4a_6+a_1a_2a_3a_4a_6-a_2^2a_4^2a_6+\\[1mm]
&+a_3a_5a_6+a_1^2a_3a_5a_6+a_2a_4a_5a_6+a_1a_2a_4a_5a_6-\\[1mm]
&-a_1a_5^2a_6-a_2a_6^2+2a_1a_2a_6^2-a_1^2a_2a_6^2,
\end{array}$$
the nullity of the following expressions are invariant
$$\begin{array}{rl}
a_3'^2-a_3'a_5'+a_5'^2-a_2'a_6'&=\displaystyle\frac{(a_3^2-a_3a_5+a_5^2-a_2a_6)(Q_{n-2}R_n-Q_nR_{n-2})^2}{P_1Q_{n-2}+a_2P_{n-2}Q_{n-2}+a_5P_nQ_{n-2}+a_3P_{n-2}Q_n+a_6P_nQ_n},\\[4mm]
a_3'a_5'-a_2'a_6'&=\displaystyle\frac{(a_3a_5-a_2a_6)(Q_{n-2}R_n-Q_nR_{n-2})^2}{P_1Q_{n-2}+a_2P_{n-2}Q_{n-2}+a_5P_nQ_{n-2}+a_3P_{n-2}Q_n+a_6P_nQ_n},\\[4mm]
\Delta'&=\displaystyle\frac{\Delta P_1^2(Q_{n-2}R_n-Q_nR_{n-2})^4}{(a_2Q_{n-2}R_{n-2}+a_5Q_nR_{n-2}+a_3Q_{n-2}R_n+a_6Q_nR_n)^2},\\[4mm]
a_3'-a_5'&=\displaystyle\frac{(a_3-a_5)(Q_{n-2}R_n-Q_nR_{n-2})}{P_1Q_{n-2}+a_2P_{n-2}Q_{n-2}+a_5P_nQ_{n-2}+a_3P_{n-2}Q_n+a_6P_nQ_n}.
\end{array}$$

We can distinguish the following non isomorphic cases:

\

\noindent \fbox{\textbf{Case 2.1.}} Let $a_3a_5-a_2a_6\neq0$ be. Then, choosing
$$\begin{array}{rl}
P_{n-2}&=-(a_6P_1Q_nR_{n-2}+a_2a_5Q_{n-2}R_{n-2}^2+a_5^2Q_nR_{n-2}^2-a_6P_1Q_{n-2}R_n+\\[1mm]
&+a_3a_5Q_{n-2}R_{n-2}R_n+a_2a_6Q_{n-2}R_{n-2}R_n+2a_5a_6Q_nR_{n-2}R_n+a_3a_6Q_{n-2}R_n^2+\\[1mm]
&+a_6Q_nR_n^2)\displaystyle\frac{1}{(a_3a_5-a_2a_6)(Q_{n-2}R_n-Q_nR_{n-2})}\\[4mm]
P_n&=-(-a_3P_1Q_nR_{n-2}-a_2^2Q_{n-2}R_{n-2}-a_2a_5Q_nR_{n-2}^2+a_3P_1Q_{n-2}R_n-\\[1mm]
&-2a_2a_3Q_{n-2}R_{n-2}R_n-a_3a_5Q_nR_{n-2}R_n-a_2a_6Q_nR_{n-2}R_n-a_3^2Q_{n-2}R_n^2-\\[1mm]
&-a_3a_6Q_nR_n^2)\displaystyle\frac{1}{(a_3a_5-a_2a_6)(Q_{n-2}R_n-Q_nR_{n-2})}
\end{array}$$
and using the restriction $(2)$, we obtain $a_3'=1.$

\

\noindent $a)$ Let $a_3-a_5\neq0$ be. Then, choosing
$$Q_n=-\frac{a_2Q_{n-2}R_{n-2}+a_5Q_{n-2}R_n}{a_3R_{n-2}+a_6R_n}, \quad R_{n-2}=\frac{a_3Q_{n-2}-a_5Q_{n-2}-a_6R_n}{a_3}$$
we get $a_5'=0, \ a_6'=1$ and $a_2'=\lambda\in \mathbb{C}\setminus\{0\}.$
The determinant of the change of basis is formed by the potencies of the following non-zero factors: $$(a_3R_{n-2}+a_6R_n)P_1Q_{n-2}(a_2R_{n-2}^2+a_3R_{n-2}R_n+a_5R_{n-2}R_n+a_6R_n^2).$$

\begin{itemize}
\item[$a.1)$] Let $\Delta\neq0$ be. Then, we choose
$$\begin{array}{rl}
R_n&=-\displaystyle\frac{(a_2a_3a_4-a_3a_5-a_2a_4a_5+a_1a_5^2+a_2a_6-a_1a_2a_6)P_1}{(a_3-a_5)(a_3a_5-a_2a_6)},\\[4mm]
Q_{n-2}&=(a_3^2a_4-a_1a_3^2a_6-2a_2a_3a_4a_6+a_3a_5a_6+a_1a_3a_5a_6+\\[1mm]
&+a_2a_4a_5a_6-a_1a_5^2a_6-a_2a_6^2+a_1a_2a_6^2)\displaystyle\frac{P_1}{(a_3-a_5)^2(a_3a_5-a_2a_6)},
\end{array}$$
and we get $a_1'=a_4'=0$ and the family $Z_8(0,\lambda,1,0,0,1),\lambda\in \mathbb{C}\setminus\{0\}.$
The determinant of change of basis is formed by the non-zero potencies of the following factors $(a_3-a_5)(a_3a_5-a_2a_6)\Delta P_1.$

\

\item[$a.2)$] Let $\Delta=0$ be. Then, we have
$$\begin{array}{l}
\Delta'=a_5'=0, \quad a_3'=a_6'=1,\\[2mm]
a_4'-a_1'-3a_2'a_4'+a_2'a_1'a_4'-a_2'^2a_4'^2-a_2'(a_1'-1)^2=0.
\end{array}$$

Thus, we obtain the following family
$$\begin{array}{ll}
e_i\circ e_j=C_{i+j-1}^j e_{i+j}, &2\leq i+j\leq n-3,\\
e_1\circ e_{n-2}=e_{n-1}, &\\
e_{n-2}\circ e_1=\alpha e_{n-1},&\\
 e_{n-2}\circ e_n=e_{n-1},& \\
 e_{n-2}\circ e_{n-2}=\beta e_{n-1}, &\mbox{with }\beta\neq0\\
e_n\circ e_n=e_{n-1},& \\
e_n\circ e_1=\gamma e_{n-1},&\mbox{with }\gamma-\alpha-3\beta\gamma+\alpha\beta\gamma-\beta^2\gamma^2-\beta(1-\alpha)^2=0
\end{array}$$

Now, we make the generic change of basis
$$\begin{array}{rl}
e_1'&=P_1e_1+P_{n-2}e_{n-2}+P_ne_n,\\
e_{n-2}'&=Q_1e_1+Q_{n-2}e_{n-2}+Q_ne_n,\\
e_n'&=R_1e_1+R_{n-2}e_{n-2}+R_ne_n.
\end{array}$$
and we have the expressions of the new parameters and the new restrictions:
$$\begin{array}{l}
\alpha'=\displaystyle\frac{\alpha P_1Q_{n-2}+\beta P_{n-2}Q_{n-2}+\gamma P_1Q_n+P_nQ_{n-2}+P_{n-2}Q_n+P_nQ_n}{P_1Q_{n-2}+\beta P_{n-2}Q_{n-2}+P_{n-2}Q_n+P_nQ_{n-2}+P_nQ_n},\\[4mm]
\beta'=\displaystyle\frac{\beta Q_{n-2}^2+2Q_{n-2}Q_n+Q_n^2}{P_1Q_{n-2}+\beta P_{n-2}Q_{n-2}+P_{n-2}Q_n+P_nQ_{n-2}+P_nQ_n},\\[4mm]
\gamma'=\displaystyle\frac{\alpha P_1R_{n-2}+\beta P_{n-2}R_{n-2}+\gamma P_1R_n+P_nR_{n-2}+P_{n-2}R_n+P_nR_n}{P_1Q_{n-2}+\beta P_{n-2}Q_{n-2}+P_{n-2}Q_n+P_nQ_{n-2}+P_nQ_n},\\[4mm]
1=\displaystyle\frac{\beta Q_{n-2}R_{n-2}+Q_{n-2}R_n+Q_nR_{n-2}+Q_nR_n}{P_1Q_{n-2}+\beta P_{n-2}Q_{n-2}+P_{n-2}Q_n+P_nQ_{n-2}+P_nQ_n},\\[4mm]
1=\displaystyle\frac{\beta R_{n-2}^2+2R_{n-2}R_n+R_n^2}{P_1Q_{n-2}+\beta P_{n-2}Q_{n-2}+P_{n-2}Q_n+P_nQ_{n-2}+P_nQ_n},\\[4mm]
0= P_1R_{n-2}+\beta P_{n-2}R_{n-2}+P_{n-2}R_n+P_nR_{n-2}+P_nR_n=0. \ (**)
\end{array}$$

Putting
$$\begin{array}{rl}
P_{n-2}&=-\displaystyle\frac{P_1Q_nR_{n-2}-P_1Q_{n-2}R_n+\beta Q_{n-2}R_{n-2}R_n+Q_{n-2}R_n^2+Q_nR_n^2}{\beta(Q_nR_{n-2}-Q_{n-2}R_n)},\\[2mm]
P_n&=(P_1Q_nR_{n-2}+\beta^2Q_{n-2}R_{n-2}^2-P_1Q_{n-2}R_n+2\beta Q_{n-2}R_{n-2}R_n+\\[1mm]
&+\beta Q_nR_{n-2}R_n+Q_{n-2}R_n^2+Q_nR_n^2)\displaystyle\frac{1}{\beta(Q_nR_{n-2}-Q_{n-2}R_n)},\\[2mm]
Q_{n-2}&=R_{n-2}+R_n,\\[2mm]
Q_n&=-\d\frac{\beta Q_{n-2}R_{n-2}}{R_{n-2}+R_n}=-\beta R_{n-2},
\end{array}$$
we get
$$\begin{array}{rl}
\alpha'&=-(-P_1R_{n-2}+2\beta P_1R_{n-2}-\alpha\beta P_1R_{n-2}+\beta^2\gamma P_1R_{n-2}+\\[2mm]
&+\beta R_{n-2}^2-\beta^2R_{n-2}^2-P_1R_n+\beta P_1R_n-\alpha\beta P_1R_n+\\[2mm]
&+R_{n-2}R_n-\beta R_{n-2}R_n+R_n^2-\beta R_n^2)\d\frac{1}{\beta(\beta R_{n-2}^2+R_{n-2}R_n+R_n^2)},\\[4mm]
\gamma'&=(P_1R_{n-2}-\beta P_1R_{n-2}+\alpha\beta P_1R_{n-2}-\beta R_{n-2}^2+P_1R_n+\\[2mm]
&+\beta\gamma P_1R_n-R_{n-2}R_n-R_n^2)\d\frac{1}{\beta(\beta R_{n-2}^2+R_{n-2}R_n+R_n^2)},\\[2mm]
\beta'&=\beta\neq0
\end{array}$$
with $P_1(\beta R_{n-2}^2+R_{n-2}R_n+R_n^2)\neq0.$
It is easy to prove that:
$$\gamma'-\alpha'-3\beta'\gamma'+\alpha'\beta'\gamma'-\beta'^2\gamma'^2-\beta'(\alpha'-1)^2=
\frac{\gamma-\alpha-3\beta\gamma+\alpha\beta\gamma-\beta^2\gamma^2-\beta(\alpha-1)^2}{\beta R_{n-2}^2+R_{n-2}R_n+R_n^2}P_1^2=0.
$$

Now, if we choose
$$\begin{array}{ll}
R_n&=\d\frac{1}{2}(P_1+\beta\gamma P_1-R_{n-2})-\\[2mm]
&-\d\frac{\sqrt{(P_1+\beta\gamma P_1-R_{n-2})^2+4(P_1R_{n-2}-\beta P_1R_{n-2}+\alpha\beta P_1R_{n-2}-\beta R_{n-2}^2)}}{2}
\end{array}$$
we get $\gamma'=0, \ \beta'=\beta\neq 0,\ \alpha'+\beta'(\alpha'-1)^2=0$ with $\alpha'\in \mathbb{C}\setminus \{0,1\}$, thus $\beta'=-\d\frac{\alpha'}{(\alpha'-1)^2}$ and we have the family $Z_9(\alpha,-\frac{\alpha'}{(\alpha'-1)^2},1,0,0,1),$ with $\alpha\in \mathbb{C}\setminus\{0,1\}.$
\end{itemize}

\


\noindent $b)$ Let $a_3-a_5=0$ be. Then, we have $a_3'=a_5'=1, \ a_3^2-a_2a_6\neq0$ and
$$\begin{array}{l}
a_1'=\d\frac{((a_1-1)Q_{n-2}+a_4Q_n)P_1}{a_2Q_{n-2}R_{n-2}+a_3Q_nR_{n-2}+a_3Q_{n-2}R_n+a_6Q_nR_n}+1,\\[3mm]
a_2'=\d\frac{a_2Q_{n-2}^2+2a_3Q_{n-2}Q_n+a_6Q_n^2}{a_2Q_{n-2}R_{n-2}+a_3Q_nR_{n-2}+a_3Q_{n-2}R_n+a_6Q_nR_n},\\[3mm]
a_4'=\d\frac{((a_1-1)R_{n-2}+a_4R_n)P_1}{a_2Q_{n-2}R_{n-2}+a_3Q_nR_{n-2}+a_3Q_{n-2}R_n+a_6Q_nR_n},\\[3mm]
a_6'=\d\frac{a_2R_{n-2}^2+2a_3R_{n-2}R_n+a_6R_n^2}{a_2Q_{n-2}R_{n-2}+a_3Q_nR_{n-2}+a_3Q_{n-2}R_n+a_6Q_nR_n},
\end{array}$$
with
$$\begin{array}{l}
P_1(Q_nR_{n-2}-Q_{n-2}R_n)(a_2Q_{n-2}R_{n-2}+a_3Q_nR_{n-2}+a_3Q_{n-2}R_n+a_6Q_nR_n)\neq 0,\\
P_1R_{n-2}+a_2P_{n-2}R_{n-2}+a_3P_{n-2}R_n+a_5P_nR_{n-2}+a_6P_nR_n=0.
\end{array}$$

We can suppose \fbox{$a'_2=0$},
\begin{itemize}
\item $a_2\neq0,$ if we choose $Q_{n-2}=\d\frac{-a_3\pm\sqrt{a_3^2-a_2a_6}}{a_2}Q_n,$ we get $a_2'=0,$
\item $a_2=0,$ if we choose $Q_{n-2}=-\d\frac{a_6Q_n}{2a_3}$ we have $a_2'=0.$
\end{itemize}

Analogously, we can suppose that \fbox{$a'_4=0$}, using $R_n$.

Now, we have the new family:
$$\begin{array}{ll}
e_i\circ e_j=C_{i+j-1}^j e_{i+j},& 2\leq i+j\leq n-3,\\
e_1\circ e_{n-2}=e_{n-1}, &\\
e_{n-2}\circ e_1=a_1' e_{n-1},&\\
e_{n-2}\circ e_n=e_{n-1}, &\\
 e_n\circ e_{n-2}=e_{n-1},& \\
e_n\circ e_n=a_6'e_{n-1}.&
\end{array}$$
and we make a generic change of basis. Choosing $P_n$ and $P_{n-2}$ (as in previous cases) we get $a_3''=a_5''=1.$
Putting $R_{n-2}=0, \ Q_{n-2}=-\d\frac{a_6Q_n}{2}$  we have $a_4''=a_2''=0$ and
$$a_1''=\frac{(1-a_1')P_1+R_n}{R_n}, \qquad a''_6=\frac{2R_n}{Q_n}$$

It is easy to check that the nullity of $a''_1-1$ is invariant because $$a_1''-1=\frac{(a_1'-1)P_1Q_{n-2}}{(Q_{n-2}+a_6Q_n)R_n}$$
Moreover, choosing $ Q_n=2R_n$ we obtain $a'_6=1.$ The determinant of change of basis is formed by the non-zero potencies of the following factors
$a_6P_1Q_nR_n.$

Now, we can distinguish two cases:
\begin{itemize}
\item[$b.1)$] If $a_1'-1\neq0,$ choosing $R_n=(a_1'-1)P_1$ we get $a_1''=0$ and the algebra $Z_{10}(0,0,1,0,1,1).$

\

\item[$b.2)$] If $a_1'-1=0,$ then $a_1''=1$ and we have $Z_{11}(1,0,1,0,1,1).$
\end{itemize}


\

\noindent \fbox{\textbf{Case 2.2.}} Let $a_3a_5-a_2a_6=0$ be.

\

As $a_3\neq 0\Rightarrow a_5=\frac{a_2a_6}{a_3}.$ We substitute in $(1)$ and in $(2)$ and we choose
$$P_{n-2}=-\frac{a_3P_1R_{n-2}+a_2a_6P_nR_{n-2}+a_3a_6P_nR_n}{a_3(a_2R_{n-2}+a_3R_n)}.$$

Now, taking into account the new parameters $(1)$, we get to:
$$\begin{array}{l}
a_2'=\d\frac{(a_2Q_{n-2}+a_3Q_n)(a_3Q_{n-2}+a_6Q_n)(a_2R_{n-2}+a_3R_n)}{a_3^2P_1(Q_{n-2}R_n-Q_nR_{n-2})},\\[3mm]
a_3'=\d\frac{(a_3Q_{n-2}+a_6Q_n)(a_2R_{n-2}+a_3R_n)^2}{a_3^2P_1(Q_{n-2}R_n-Q_nR_{n-2})}\neq0.
\end{array}$$
with $a_3P_1(Q_nR_{n-2}-Q_{n-2}R_n)(a_2R_1+a_3R_n)\neq 0$ and that the nullity of the following expressions:
$$\begin{array}{l}
a_1'a_6'-a_3'a_4'=\d\frac{a_3(a_1a_6-a_3a_4)(a_2R_{n-2}+a_3R_n)(Q_{n-2}R_n-Q_nR_{n-2})P_1}{(a_3a_6P_nQ_n+a_3P_1Q_{n-2}+a_2a_3P_{n-2}Q_{n-2}+a_2a_6P_nQ_{n-2}+a_3^2P_{n-2}Q_n)^2},\\[3mm]
a_3'^2-a_2'a_6'=\d\frac{(a_3^2-a_2a_6)(a_3Q_{n-2}+a_6Q_n)(a_2R_{n-2}+a_3R_n)(Q_{n-2}R_n-Q_nR_{n-2})}{(a_3a_6P_nQ_n+a_3P_1Q_{n-2}+a_2a_3P_{n-2}Q_{n-2}+a_2a_6P_nQ_{n-2}+a_3^2P_{n-2}Q_n)^2},
\end{array}$$
are invariant. Thus, we can distinguish the non isomorphic cases:

\

\noindent $a)$ Let $a_3^2-a_2a_6\neq 0$ be. Then, choosing
$$\begin{array}{l}
Q_n=-\d\frac{a_2Q_{n-2}}{a_3},\\
\qquad \ \Rightarrow a'_2=a'_5=0\\[2mm]
P_n=\d\frac{(a_2a_3^3P_1R_{n-2}-a_1a_2a_3^3P_1R_{n-2}+a_2^2a_3^2a_4P_1R_{n-2}-a_2^2a_3a_6P_1R_{n-2}
-a_1a_3^4P_1R_n+a_2a_3^3a_4P_1R_n)}{(a_3^2-a_2a_6)^2(a_2R_{n-2}+a_3R_n)},\\
\qquad \ \Rightarrow a'_1=0,\\[2mm]
R_{n-2}=-\d\frac{a_6R_n}{a_3},\\
\qquad \ \Rightarrow a'_6=0\\[2mm]
P_1=\d\frac{(a_3^2-a_2a_6)^2R_n}{a_3^3},\\
\qquad \ \Rightarrow a'_3=1.
\end{array}$$
it is easy to check that the determinant of the change of basis
is formed by the potencies of the following non-zero factors $a_3(a_3^2-a_2a_6)P_1Q_{n-2}R_n.$
We compute the new parameter $a_4:$
$$a_4'=\frac{(a_3a_4-a_1a_6)R_n}{a_3Q_1}$$

\begin{itemize}
\item[$a.1)$] If $a_3a_4-a_1a_6\neq0,$ then choosing $Q_{n-2}=\d\frac{(a_3a_4-a_1a_6)R_n}{a_3},$ we have $a_4'=1$ and $Z_{12}(0,0,1,1,0,0).$

\item[$a.2)$] If $a_3a_4-a_1a_6=0,$ then  $a_4'=0$ and we have $Z_{13}(0,0,1,0,0,0).$
\end{itemize}

\

\noindent $b)$ Let $a_3^2-a_2a_6=0$ be. We have that $a_3\neq0,\ a_2\neq0$ and $a_6\neq0.$ It implies that $a_6=\frac{a_3^2}{a_2}$ and $a_5=\frac{a_2a_6}{a_3}=a_3.$ As in case $a)$, we choose $$P_n=-\frac{a_2(P_1R_{n-2}+a_2P_{n-2}R_{n-2}+a_3P_{n-2}R_n)}{a_3(a_2R_{n-2}+a_3R_n)}.$$

It is easy to check that the nullity of the expressions
$$a_3'-a_1'a_3'+a_2'a_4'=\frac{(a_3-a_1a_3+a_2a_4)(a_2Q_{n-2}+a_3Q_n)(a_2R_{n-2}+a_3R_n)^2}{a_2a_3^2P_1(Q_{n-2}R_n-Q_nR_{n-2})}$$
is invariant.

Choosing
$$\begin{array}{l}
P_1=-\d\frac{(a_2Q_{n-2}+a_3Q_n)(a_2R_{n-2}+a_3R_n)^2}{a_2a_3(Q_nR_{n-2}+Q_{n-2}R_n)}\Rightarrow a_3'=a_5'=1\\[2mm]
R_n=\d\frac{a_2Q_{n-2}+a_3Q_n-a_2R_{n-2}}{a_3}\Rightarrow a'_2=a'_6=1.
\end{array}$$
we get
$$\begin{array}{l}
a_1'=\d\frac{a_1Q_{n-2}+a_4Q_n-R_{n-2}}{Q_{n-2}-R_{n-2}},\\[2mm]
a_4'=\d\frac{a_2a_4Q_{n-2}+a_3a_4Q_n-a_3(1-a_1)R_{n-2}-a_2a_4R_{n-2}}{a_3(Q_{n-2}-R_{n-2})}.
\end{array}$$
with $a_2a_3(a_2Q_{n-2}+a_3Q_n)(Q_{n-2}-R_{n-2})\neq0.$

Thus, we can distinguish two non isomorphic cases:

\begin{itemize}
\item[$b.1)$] Let $a_3-a_1a_3+a_2a_4\neq0$ be.

\

If $a_4\neq0,$ then choosing $$Q_n=\frac{-a_2a_4Q_{n-2}+a_3R_{n-2}-a_1a_3R_{n-2}+a_2a_4R_{n-2}}{a_3a_4}$$
we have $a_4'=0$ and $a_1'=\frac{a_1a_3-a_2a_4}{a_3}\neq 1,$ that is, $a_1'=\lambda\in \mathbb{C}\setminus \{1\}.$
The determinant of the change of basis is formed by the potencies of the non-zero factors $a_2a_3a_4(a_3-a_1a_3+a_2a_4)R_{n-2}(Q_{n-2}-R_{n-2}).$
We obtain $Z_{14}=(\lambda,1,1,0,1,1), \lambda\in \mathbb{C}\setminus \{1\}.$

\

If $a_4=0,$ then choosing $R_{n-2}=0$ and $R_{n-2}\neq Q_{n-2},$ we have $a_4'=0.$
The determinant of the change of basis is formed by the potencies of the non-zero factors $a_2a_3Q_{n-2}(a_2Q_{n-2}+a_3R_{n}).$
We get to the previous family.

\item[$b.2)$] Let $a_3-a_1a_3+a_2a_4=0$ be.

\

We can suppose $a_4=\d\frac{(a_1-1)a_3}{a_2}$. It is easy to prove  that the nullity of
$$a_1'-1=\d\frac{(a_1-1)(a_2Q_{n-2}+a_3Q_n)(a_2R_{n-2}+a_3R_n)}{a_2a_3(Q_{n}R_{n-2}-Q_{n-2}R_n)}$$
is invariant.

Putting the adequate values of the parameters $P_n,P_1$ and $R_n,$ we get to $a_2'=a_3'=a_5'=a_6'=1$ and
$$a_1'=\frac{a_1a_2Q_{n-2}-a_3Q_n+a_1a_3Q_n-a_2R_{n-2}}{a_2a_3(Q_nR_{n-2}-Q_{n-2}R_n}$$
with $a_2a_3(a_2Q_{n-2}+a_3Q_n)(Q_{n-2}-R_{n-2})\neq 0.$


Now, we can distinguish two cases:

\

\begin{itemize}
\item If $a_1-1\neq0$, choosing $Q_n=-\d\frac{a_2(a_1Q_{n-2}-R_{n-2})}{(a_1-1)a_3},$ we have $a_1'=0,\ a_4'=-1$ and the algebra $Z_{15}(0,1,1,-1,1,1).$

\

\item If $a_1-1=0,$ we have $a_1'=a_4'=1$ and $Z_{16}(1,1,1,0,1,1).$
\end{itemize}
\end{itemize}
The theorem is proved.
\end{proof}


\

\noindent \textbf{Case II.} If $r_1\geq 3,\ r_2=1$,  then we have the following gradation:
$$ Z_1= \langle e_1,e_n \rangle,\ Z_2=\langle e_2\rangle, \dots, Z_{r_1}=\langle
e_{r_1},e_{n-2}\rangle,\ Z_{r_1+1}=\langle e_{r_1+1},e_{n-1}\rangle,\dots, Z_{n-3}=\langle e_{n-3}\rangle$$

Let $Z$ be a $n$-dimensional naturally graded Zinbiel algebra of type I with $r_1\geq 3$ and $r_2=1,$ then the following lemma is true.

\begin{lem}
If $r_1\geq 3,$ $ r_2=1$ and $n\geq 7,$ there are not any naturally graded Zinbiel algebras.
\end{lem}

\begin{proof}
According to the properties of the gradation, $Z_1\circ Z_1=Z_2,$ we have:
$$\begin{array}{l}
e_n \circ e_1 =\alpha_1e_2,\\
e_n \circ e_n= \alpha_2e_2.
\end{array}$$

From $Z(e_1,e_n,e_1) =Z(e_1,e_n,e_n) =0$ we get $\alpha_1=\alpha_2=0.$

Furthermore, we have
$$e_n\circ e_i=e_n\circ (e_1\circ e_{i-1})=(e_n\circ e_1)\circ e_{i-1}-(i-1)e_n\circ e_i
\Rightarrow e_n\circ e_i=0\ \mbox{ for }\ 1\leq i\leq n-3,$$
and
$$e_i \circ e_n= (e_1 \circ e_{i-1}) \circ e_n = e_1 \circ (e_{i-1} \circ
e_n)+e_1 \circ (e_n \circ e_{i-1}) = 0 \Rightarrow e_i \circ e_n =
0\ \mbox{ for }\ 1\leq i \leq n-3.$$

We observe that it is not possible to obtain the element $e_{n-2}$ for $n\geq 7$, this contradicts with the supposition $r_1\geq 3$. Thus, in this case, we do not obtain any naturally graded Zinbiel algebra.
\end{proof}

\


\noindent \textbf{Case III.} If $r_1=2,\ r_2=1$, then
$$ Z_1= \langle e_1,e_n
\rangle,\ Z_2=\langle e_2,e_{n-2} \rangle,\ Z_{3}=\langle
e_3,e_{n-1}\rangle,\ Z_4=\langle e_4\rangle,\dots, Z_{n-3}=\langle e_{n-3}\rangle $$

Let $Z$ be a $n$-dimensional naturally graded Zinbiel algebra of type I with $r_1\geq 2$ and $r_2=1,$ then the following lemma is true.

\begin{lem}
If $r_1=2,$ $r_2=1$ and $n\geq 8,$ then there are not any naturally graded Zinbiel algebra.
\end{lem}

\begin{proof}
According to the properties of the gradation, $Z_1\circ Z_1=Z_2,$ we obtain the following multiplication:
$$\begin{array}{l}
e_n\circ e_1=\alpha_1e_2+\beta_1e_{n-2},\\
e_n\circ e_n =\alpha_2e_2+\beta_2e_{n-2}.
\end{array}$$

From the identity $Z(e_1,e_n,e_1)=Z(e_1,e_n,e_n)=0,$  $\alpha_1 =\beta_1=\alpha_2 =\beta_2 = 0$ follows.

As in the previous case, we prove that $e_n\circ e_i=e_i\circ e_n=0$ for $1\leq i\leq n-3$, and therefore it is not possible to obtain the the element $e_{n-2}$ for $n\geq8$ and it contradicts the gradation.
Thus, in this case, we do not obtain any naturally graded Zinbiel algebra.
\end{proof}

\


\noindent \textbf{Case IV.} If $r_1=1,\ r_2\geq 4$, then
$$ Z_1= \langle
e_1,e_{n-2} \rangle,\ Z_2=\langle e_2,e_{n-1}\rangle,\dots,
Z_{r_1}=\langle e_{r_1},e_n\rangle,\dots,Z_{n-3}=\langle e_{n-3}\rangle$$

Let $Z$ be a $n$-dimensional naturally graded Zinbiel algebra of type I with $r_1=1$ and $r_2\geq 4,$ then the following lemma is true.

\begin{lem}
If $r_1=1,$ $r_2\geq 4$ and $n\geq 8,$ there are not any naturally graded Zinbiel algebras.
\end{lem}

\begin{proof}
Similar to Case II and Case III.
%
%
%
\end{proof}

\

\noindent \textbf{Case V.} If $r_1=1,$ $ r_2=2$, then
$$ Z_1= \langle
e_1,e_{n-2} \rangle,\ Z_2=\langle e_2,e_{n-1},e_n\rangle,\dots,
Z_{3}=\langle e_{3}\rangle,\dots,Z_{n-3}=\langle e_{n-3}\rangle$$

Let $Z$ be a $n$-dimensional naturally graded Zinbiel algebra of type I with $r_1=1,$ $r_2=2$ and $n\geq 8.$
We have
$$\begin{array}{lll}
 e_{n-2}\circ e_1 =\alpha_1e_2+ \beta_1e_{n-1}+\gamma_1e_n,\\[1mm]
 e_{n-2}\circ e_{n-2}=\alpha_2e_2+ \beta_2e_{n-1}+\gamma_2e_n,\\[1mm]
 e_{n-2}\circ e_{n-1}=\delta_1e_3,\\[1mm]
 e_{n-2}\circ e_n=\delta_2e_3,\\[1mm]
 e_n\circ e_1=\delta_3e_3, \\[1mm]
 e_n\circ e_{n-2}=\delta_4e_3,\\[1mm]
 e_n\circ e_{n-1}=\delta_5e_3,\\[1mm]
 e_n\circ e_n=\delta_6e_3.
\end{array}$$

We compute the following identities of Zinbiel
$$\begin{array}{lll}
Z(e_1,e_{n-2},e_1) = Z(e_1,e_{n-1},e_1) = Z(e_1,e_1,e_{n-2})=Z(e_{n-2},e_1,e_1)=0\\
Z(e_{n-2},e_{n-1},e_1) =Z(e_{n-2},e_{n-2},e_1) =Z(e_1,e_{n-2},e_{n-2}) =Z(e_1,e_n,e_1)=0\\
Z(e_{n-2},e_n,e_1) = Z(e_1,e_n,e_{n-2}) =Z(e_1,e_{n-2},e_{n-1}) =Z(e_1,e_{n-2},e_n) =0\\
Z(e_n,e_{n-1},e_1) =Z(e_1,e_n,e_n) =0.
\end{array}$$
and we have $\alpha_1=\alpha_2=\delta_1=\delta_2=\delta_3=\delta_4=\delta_5=\delta_6=0,$
and $$e_2\circ e_{n-2}=e_{n-2}\circ e_2=e_{n-1}\circ e_{n-2}=e_{n-1}\circ
e_{n-1}=e_{n-1}\circ e_n=0.$$

Now we will consider the following multiplication:
$$\begin{array}{rll}
e_{n-1}\circ e_i&=(e_{n-1}\circ e_1)\circ
e_{i-1}-(i-1)e_{n-1}\circ e_i&\\[1mm]
 \Rightarrow &e_{n-1}\circ e_i=0,& 1\leq i\leq n-3,\\[1mm]

e_i \circ e_{n-1}& = (e_1 \circ e_{i-1}) \circ e_{n-1} =&\\[1mm]
&= e_1 \circ(e_{i-1} \circ e_{n-1})+e_1 \circ (e_{n-1} \circ e_{i-1}) = 0&\\[1mm]
\Rightarrow &e_i \circ e_{n-1} = 0,&1\leq i \leq n-3,\\[1mm]
%
e_{n-2}\circ e_i&=(e_{n-2}\circ e_1)\circ
e_{i-1}-(i-1)e_{n-2}\circ e_i&\\[1mm]
 \Rightarrow &e_{n-2}\circ e_i=0,& 2\leq i\leq n-3,\\[1mm]
e_i \circ e_{n-2} &= (e_1 \circ e_{i-1}) \circ e_{n-2} =&\\[1mm]
&= e_1 \circ
(e_{i-1} \circ e_{n-2})+e_1 \circ (e_{n-2} \circ e_{i-1}) = 0,&\\[1mm]
\Rightarrow& e_i \circ e_{n-2} = 0, &2\leq i \leq n-3,\\[1mm]

e_{n}\circ e_i&=(e_n\circ e_1)\circ e_{i-1}-(i-1)e_n\circ e_i&\\[1mm]
\Rightarrow & e_n\circ e_i=0,& 1\leq i\leq n-3,\\[1mm]

e_i \circ e_n& = (e_1 \circ e_{i-1}) \circ e_n =&\\[1mm]
&=e_1 \circ (e_{i-1} \circ e_n)+e_1 \circ (e_n \circ e_{i-1}) = 0&\\[1mm]
\Rightarrow&e_i \circ e_n = 0,&1\leq i \leq n-3,
\end{array}$$

Thus, we have received the following family:
$$Z(\beta_1,\beta_2,\gamma_1,\gamma_2):\left\{\begin{array}{ll}
e_i\circ e_j= C_{i+j-1}^j e_{i+j}, & 2\leq i+j\leq n-3,\\[1mm]
e_1\circ e_{n-2}=e_{n-1},&\\[1mm]
e_{n-2}\circ e_1=\beta_1e_{n-1}+\gamma_1e_n,&\\[1mm]
e_{n-2}\circ e_{n-2}=\beta_2e_{n-1}+\gamma_2e_n,& \mbox{ with }(\gamma_1,\gamma_2)\neq (0,0).
\end{array}\right.$$


\

\begin{thm} Let $Z$ be an $n$-dimensional naturally graded Zinbiel algebra of type I with $r_1=1,$ $r_2=2$ and $n\geq 8.$ Then, $Z$ is isomorphic to one of the following algebras, pairwise non isomorphic:
$$Z_{17}(\lambda,0,0,1),\ \lambda \in \mathbb{C}, \quad Z_{18}(1,0,1,1), \quad Z_{19}(0,1,1,0),\quad Z_{20}(0,0,1,0).$$
\end{thm}

\begin{proof} As in Theorem \ref{teorema1} we make the generic change of basis and we study all de cases.
\end{proof}

\


\noindent \textbf{Case VI.} If $r_1=1,\ r_2=3$, then
$$ Z_1= \langle
e_1,e_{n-2} \rangle,\ Z_2=\langle e_2,e_{n-1}\rangle,\dots,
Z_{3}=\langle e_{3},e_n\rangle,\dots,Z_{n-3}=\langle e_{n-3}\rangle$$

Let $Z$ be an $n$-dimensional naturally graded Zinbiel algebra of type I with $r_1=1$, $r_2=3$ and $n\geq 8.$

\begin{thm}
Let $Z$ be an $n$-dimensional naturally graded Zinbiel algebra of type I with $r_1=1,$ $r_2=3$ and $n\geq 8.$ Then, $Z$ is isomorphic to the following algebra:
$$Z_{21}:\left\{\begin{array}{ll}
e_i\circ e_j= C_{i+j-1}^j e_{i+j}, & 2\leq i+j\leq n-3,\\[1mm]
e_1\circ e_{n-2}=e_{n-1},&\\[1mm]
e_{n-2}\circ e_1=-e_{n-1},&\\[1mm]
e_{n-2}\circ e_{n-1}=e_n.&
\end{array}\right.$$

\end{thm}

\begin{proof}
Using the gradation and its properties, we compute $Z_1\circ Z_1,$ $Z_1\circ Z_2$ and $Z_2\circ Z_1$ and we have:
$$\begin{array}{l}
e_{n-2}\circ e_1=\alpha_1e_2+\beta_1e_{n-1},\\
e_{n-2}\circ e_{n-2}=\alpha_2e_2+\beta_2e_{n-1},\\
e_{n-2}\circ e_{n-1}=\alpha_3e_3+\beta_3e_n,\\
e_{n-1}\circ e_{n-2}=2\alpha_2e_3,\\
e_{n-2}\circ e_2=\alpha_4e_3+\beta_4e_n,\\
e_{n-1}\circ e_1=\alpha_1e_3,\\
e_2\circ e_{n-2}=\alpha_1e_3.
\end{array}$$

From the following Zinbiel identities:
$$Z(e_1,e_{n-1},e_1)=Z(e_{n-2},e_1,e_1)=Z(e_{n-2},e_{n-1},e_1)=Z(e_{n-2},e_{n-2},e_1)=Z(e_{n-2},e_{n-2},e_2)=0$$
we obtain that
$$(3):\left\{\begin{array}{l}
\alpha_1=\alpha_4=\beta_4=0,\\
3\alpha_3+\beta_3 (e_n\circ e_1)=0,\\
2\alpha_2-(\beta_1+1)\alpha_3=0,\\
(\beta_1+1)\beta_3=0,\\
\beta_2(e_{n-1}\circ e_2)=0.
\end{array}\right.$$

From the restrictions, it is easy to see that $\alpha_2=0.$ Moreover, we compute $e_{n-1}\circ e_2=e_2\circ e_{n-1}=0.$
Furthermore, from the multiplication $Z_1\circ Z_3$ and $Z_3\circ Z_1,$ we have
$$\begin{array}{l}
e_n\circ e_1=\delta_1e_4,\\
e_{n-2}\circ e_n=\delta_3e_4,\\
e_n\circ e_{n-2}=\delta_4e_4,\\
e_n\circ e_n=\left\{\begin{array}{ll}
\delta_2e_6,&n\geq 9,\\
0,&n=8
\end{array}\right.
\end{array}$$

Now we will consider the following equalities:
$$Z(e_1,e_n,e_1)=Z(e_{n-2},e_n,e_1)=Z(e_1,e_n,e_{n-2})=Z(e_{n-2},e_{n-2},e_{n-2})=0$$
we get $\delta_1=\delta_3=\delta_4=0.$

If $n\geq 10,$ we make $Z(e_1,e_n,e_n)=0$ and we obtain $\delta_2=0.$
Thus, $e_n\in Center(Z).$

Taking $(3)$ into account we have $\alpha_3=0.$ Thus, $e_{n-2}\circ e_{n-1}=\beta_3 e_n.$ If $\beta_3=0$ we have a contradiction of the gradation, in particular with $r_2=3.$ Thus, $\beta_3\neq 0.$ We can suppose $e_{n-2}\circ e_{n-1}=e_n.$ By using $(3),$ we can see $\beta_1=-1.$

From $Z(e_{n-2},e_1,e_{n-2})=0$ we get $\beta_2=0.$ Now, similar to previous cases we prove that
$$e_{n-1}\circ e_i=e_i\circ e_{n-1}=e_{n-2}\circ e_i=e_i\circ e_{n-2}=0,\ \mbox{ for }\ 2\leq i\leq n-3.$$

Furthermore, we have:
$$\begin{array}{l}
e_{n-1}\circ e_{n-1}=\delta_5 e_4,\\
e_{n-1}\circ e_n=\delta_6 e_5,\\
e_{n}\circ e_{n-1}=\delta_7 e_5,
\end{array}$$
and from the equalities:
$Z(e_1,e_{n-2},e_{n-1})=Z(e_1,e_{n-2},e_{n})=Z(e_{n-2},e_{n-1},e_{n-1})=0,$
we get $\delta_5=\delta_6=\delta_7=0.$

It only has to be proved that when $n=9,$ then $\delta_2=0.$ Now, from $Z(e_{n-2},e_{n-1},e_n)=0$ we lead to $\delta_2=0.$

Finally, we obtain the algebra of the theorem.
\end{proof}

\subsection{Type II}
Now we will consider naturally graded Zinbiel algebra of the second type.

Let $Z$ be a $n$-dimensional naturally graded Zinbiel algebra of type II, then there exists a basis $\{e_1,e_2,\dots,e_n\}$
such that the operator of left multiplication $L_{e_1}$ has the following matrix form:
$$\left(\begin{array}{ccc}
J_2&0&0\\
0&J_{n-3}&0\\
0&0&J_1
\end{array}\right)$$

We have the products:
$$\begin{array}{ll}
e_1\circ e_1=e_2,&\\
e_1\circ e_2=0,&\\
e_1\circ e_i=e_{i+1},&3\leq i\leq n-2,\\
e_1\circ e_{n-1}=0,&\\
e_1\circ e_n=0.&
\end{array}$$

Thus, the subspaces of the natural gradation are:
$$<e_1,e_3>\subseteq Z_1,\ <e_2,e_4>\subseteq Z_2,\ <e_5>\subseteq Z_3,\ \dots,\ <e_{n-1}>\subseteq Z_{n-3} $$

Let us assume that $e_n\in Z_{r}$ with $1\leq r\leq n-3.$

\begin{thm}
There does not exist any naturally graded Zinbiel algebra of type II with dimension greater or equal to 9.
\end{thm}

\begin{proof}
Using the property of the gradation, $Z_i\circ Z_j\subseteq Z_{i+j},$ we have that $e_3\circ e_1=\alpha_1 e_2+\beta_1 e_4$ and $e_3\circ e_2=\beta_2e_5+(*)e_n,$ where $(*)$
indicates the coefficient of the vector $e_n$ when $r=3,$ or $(*)=0.$

Let us consider the following product:
$$\begin{array}{l}
e_4\circ e_1=(e_1\circ e_3)\circ e_1=(1+\beta_1) e_5,\\
e_2\circ e_3=(e_1\circ e_1)\circ e_3=(1+\beta_1) e_5,\\
e_2\circ e_4=(e_1\circ e_1)\circ e_4=(2+\beta_1) e_6,\\
e_4\circ e_2=(e_1\circ e_3)\circ e_2=(1+\beta_1+\beta_2) e_6,\\
e_5\circ e_1=(e_1\circ e_4)\circ e_1=(2+\beta_1) e_6,\\
e_2\circ e_5=(e_1\circ e_1)\circ e_5=(3+\beta_1) e_7,\\
e_5\circ e_2=(e_1\circ e_4)\circ e_2=(3+2\beta_1+\beta_2) e_7,\\
\end{array}
$$

From the following identities:
$$\begin{array}{ll}
0=Z(e_1,e_2,e_3)&\Rightarrow e_3\circ e_3=(1+\beta_1+\beta_2) e_6,\\
                &\Rightarrow e_3\circ e_3\in Z_2 \mbox{ and } e_6\in Z_4,\\
                &\Rightarrow 1+\beta_1+\beta_2=0\\[1mm]
0=Z(e_1,e_2,e_4)&\Rightarrow e_3\circ e_4=(3+2\beta_1+\beta_2) e_7,\\
                &\Rightarrow e_3\circ e_4\in Z_3 \mbox{ and } e_7\in Z_5,\\
                &\Rightarrow 3+2\beta_1+\beta_2=0\\[1mm]
0=Z(e_1,e_2,e_5)&\Rightarrow e_3\circ e_5=(6+3\beta_1+\beta_2) e_8,\\
                &\Rightarrow e_3\circ e_5\in Z_4 \mbox{ and } e_8\in Z_6,\\
                &\Rightarrow 6+3\beta_1+\beta_2=0.
\end{array}$$

We have the next system of equations:
$$\begin{array}{l}
1+\beta_1+\beta_2=0\\
3+2\beta_1+\beta_2=0\\
6+3\beta_1+\beta_2=0\\
\end{array}$$

It is trivial to see that this system of equations do not have solution. Then, there does not exist any naturally graded Zinbiel algebra of type II with dimension greater or equal to 9.
\end{proof}

\subsection{Type III}

Now we will consider naturally graded Zinbiel algebra of type III.

Let $Z$ be a $n$-dimensional naturally graded Zinbiel algebra of type III, then there exists a basis $\{e_1,e_2,\dots,e_n\}$
such that the operator of left multiplication $L_{e_1}$ has the following matrix form:
$$\left(\begin{array}{ccc}
J_1&0&0\\
0&J_{n-3}&0\\
0&0&J_2
\end{array}\right)$$

We have the products:
$$\begin{array}{ll}
e_1\circ e_1=0,&\\
e_1\circ e_i=e_{i+1},&2\leq i\leq n-3,\\
e_1\circ e_{n-1}=e_n,&\\
e_1\circ e_n=0.&
\end{array}$$

Then, the subspaces of the natural gradation are:
$$<e_1,e_2>\subseteq Z_1,\ <e_3>\subseteq Z_2,\ <e_4>\subseteq Z_3,\ \dots,\ <e_{n-2}>\subseteq Z_{n-3} $$

\begin{thm}
There does not exist any naturally graded Zinbiel algebra of type III with dimension greater or equal to 7.
\end{thm}

\begin{proof}
Using the identity $(a\circ b)\circ c=(a\circ c)\circ b$ we have
$$e_3\circ e_1=(e_1\circ e_2)\circ e_1=(e_1\circ e_1)\circ e_2=0$$
that is, $e_3\circ e_1=0.$

From the following identity we have:
$$0=(e_1\circ e_1)\circ e_3=e_1\circ (e_1\circ e_3)+ e_1\circ (e_3\circ e_1)=e_1\circ e_4=e_5$$

We get a contradiction because $n\geq 6$. Thus, there does not exist any naturally graded Zinbiel algebra of type III.
\end{proof}

\end{document}